\documentclass[12pt,twoside]{amsart}
\usepackage{latexsym,amssymb,amsmath}

\usepackage{tikz-cd}
\usepackage{mathtools}
\usepackage[all]{xy}
\usepackage{hyperref}
\hypersetup{colorlinks=true}

\numberwithin{equation}{section}
\newtheorem{theorem}{Theorem}[section]
\newtheorem{lemma}[theorem]{Lemma}
\newtheorem{proposition}[theorem]{Proposition}

\newtheorem{theoremx}{\bf Theorem}
\newtheorem{conjecture}[theoremx]{Conjecture}

\theoremstyle{definition}
\newtheorem{definition}[theorem]{Definition} 
 
\newtheorem{remark}[theorem]{Remark}

\newtheorem{notation}[theorem]{Notation}

\newcommand{\NN}{\mathbb{N}}

\newcommand{\ZZ}{\mathbb{Z}}
\newcommand{\QQ}{\mathbb{Q}}
\newcommand{\FF}{\mathbb{F}}

\newcommand{\KK}{\mathbb{K}}

\DeclareMathOperator{\Spec}{{Spec}}

\DeclareMathOperator{\Ass}{{Ass}}

\DeclareMathOperator{\fpt}{{fpt}}

\DeclareMathOperator{\reg}{{reg}}

\DeclareMathOperator{\aaa}{{a}}
\DeclareMathOperator{\bb}{{b}}
\DeclareMathOperator{\cc}{{c}}

\DeclareMathOperator{\e}{{e}}

\newcommand{\m}{\mathfrak{m}}
\newcommand{\cF}{\mathcal{F}}
\newcommand{\p}{\mathfrak{p}}
\newcommand{\q}{\mathfrak{q}}
\newcommand{\cA}{\mathcal{A}}
\newcommand{\cV}{\mathcal{V}}
\newcommand{\cD}{\mathcal{D}}

\begin{document}


\title[Minimum
distance]{Bounds for the minimum distance function}

\author[Luis N\'u\~nez-Betancourt]{Luis N\'u\~nez-Betancourt$^{\aaa}$}
\address{Luis N\'u\~nez-Betancourt \\ Centro de Investigaci\'on en Matem\'aticas\\ Guanajuato, Gto., M\'exico.}
\thanks{$^{\aaa}$ The first author was partially supported by CONACyT Grant 284598,  Cátedras Marcos Moshinsky, and  SNI, Mexico.}
\email{luisnub@cimat.mx}

\author[Yuriko Pitones]{Yuriko Pitones$^{\bb}$}
\address{Yuriko Pitones\\ 
Centro de Investigaci\'on en Matem\'aticas\\ Guanajuato, Gto., M\'exico.}
\thanks{$^{\bb}$ The second author was partially supported by CONACyT Grant 427234 and  CONACyT Postdoctoral Fellowship 177609.}
\email{yuriko.pitones@cimat.mx}

\author[Rafael H. Villarreal]{Rafael H. Villarreal$^{\cc}$}
\address{Rafael H. Villarreal\\
Departamento de
Matem\'aticas\\
Centro de Investigaci\'on y de Estudios
Avanzados del
IPN\\
Apartado Postal
14--740 \\
07000 Mexico City, D.F.
}
\thanks{$^{\cc}$ The third author was partially supported by SNI, Mexico.}
\email{vila@math.cinvestav.mx}


\maketitle 

\parindent=8mm

\begin{abstract}
\noindent
Let $I$ be a homogeneous ideal in a polynomial ring $S$.
In this paper, we extend the study of the asymptotic behavior of the minimum distance function $\delta_I$ of $I$ and give bounds for its stabilization point, $r_I$, when $I$ is  an $F$-pure or a square-free monomial ideal.
These bounds are related with the dimension and the Castelnuovo--Mumford regularity of $I$.
\end{abstract}


\tableofcontents

\section{Introduction}\label{Intro}
In this manuscript we study the minimum distance function $\delta_I$ of a homogeneous ideal $I$ contained in a polynomial ring $S=\KK[x_1,\ldots,x_n]$ over a field $\KK$. 
This minimum distance function for ideals  was introduced by the second-named and third-named authors together with Mart\'inez-Bernal \cite{MBPV} to obtain an algebraic formulation of the minimum distance of projective Reed--Muller-type codes over finite fields.

If $I$ is an unmixed radical graded ideal and its associate primes are generated by linear forms, then $\delta_I$ is non-increasing  \cite{MBPV}.
 In our first result, we extend this property to any radical  ideal.
\begin{theoremx}[{Theorem \ref{ThmNonDec}}]\label{MainThmNonDec}
Suppose that $I\subseteq S$ is a radical ideal.
Then, $\delta_{I}(d)$ is a non-increasing function.
\end{theoremx}

The previous result allow us to define the regularity index of $I$, $r_I$, as the value where $\delta_I$ stabilizes.
If $\dim(S/I)=1$, previous work shows that $r_I\leq  \reg(S/I)$ \cite{GSRTR,RMSV}, where $\reg(S/I)$ is the Castelnuovo--Mumford regularity of $S/I$.
This motivated the authors to  conjecture that this  relation holds in greater generality.

\begin{conjecture}[\cite{NBPV}]\label{Conj}
Let $I\subseteq S$ be a radical homogeneous ideal whose associated primes are generated by linear forms. Then, $r_I\leq \reg(S/I).$
\end{conjecture}

This conjecture was previously showed for edge ideals associated to Cohen--Macaulay bipartite graphs \cite{NBPV} and if $\dim(S/I)=1$ \cite{GSRTR,RMSV}. However, the conjecture does not hold in general. Jaramillo and the third-named author  provided an example of a monomial edge ideal $I$ such that  $r_I> \reg(S/I)$ \cite{JV}.
In this work, we find bounds for $r_I$ for square-free monomial ideals.

\begin{theoremx}[{Theorem \ref{ThmShellable} \& \ref{ThmSRGor} }]\label{MainReg}
Let $I\subseteq S$ be a square-free monomial ideal.
Then, $r_I\leq \dim(S/I) $. Moreover, if $I$ is shellable or Gorenstein, then $r_I\leq \reg(S/I)$.
\end{theoremx}

We also have prove a similar result for ideals such that $S/I$ is a  $F$-pure ring.
These class of rings play an important role in the study of singularities in prime characteristic \cite{HRFpurity}.

\begin{theoremx}[{Theorem \ref{ThmFpureDim} \& \ref{ThmGorFpure}}]\label{MainReg}
Suppose that $\KK$ is a field of prime characteristic.
Let $I\subseteq S$ be an ideal such that $S/I$ if $F$-pure.
Then, $r_I\leq \dim(S/I) $. Moreover, if $I$ is  Gorenstein, then $r_I\leq \reg(S/I)$.
\end{theoremx}

\section{Preliminaries}\label{prelim-section}

In this section we recall some well known notion and results that are needed throughout this manuscript.

Let $S=\KK[x_1,\ldots,x_n]=\bigoplus_{t=0}^{\infty} S_t$ be a polynomial ring over
a field $\KK$ with the standard grading and let $I\neq(0)$ be a homogeneous  ideal
of $S$.  Let $d$ denote the Krull dimension of $R=S/I$.

The {\it Hilbert function} of $S/I$, denoted $H_I$, is given by
$$
H_I(t)=\dim_\KK(R_{\leq t})=\dim_\KK(S_{\leq t}/I_{\leq t})=\dim_\KK (S_{\leq t}) - \dim_\KK (I_{\leq t}),
$$ 
where $I_{\leq t}=I\cap S_{\leq t}$. By a classical theorem of Hilbert
there is a unique polynomial
$h_I(t)\in\mathbb{Q}[t]$ of 
degree $d$ such that $H_I(t)=h_I(t)$ for  $t\gg 0$. 

The {\it Hilbert-Samuel multiplicity\/} or {\it degree\/}  of $R$, denoted
$\e(R)$, is the 
positive integer 
defined by
$\e(R)=
d! \lim\limits_{t\rightarrow\infty}{H_I(t)}/{t^{d}} $. 

Given an integer $t\geq 1$, let $\mathcal{F}_t$ be the set of 
all zero-divisors of $S/I$ not in $I$ of degree $t\geq 1$. That is 
$$
\mathcal{F}_t=\{\, f\in S_t\, \vert\, f\not\in I,\;  (I\colon f)\neq I \}.
$$
We note that $ (I\colon f)\neq I$ is equivalent to $f\in\p$ for some $\p\in\Ass_S(S/I)$, $Ass_S(S/I)$ is the set of associated primes of $S/I$.

\begin{definition}\rm 
The {\it minimum distance function\/} of $I$ is the function  
$\delta_I\colon \mathbb{N}_+\rightarrow \mathbb{Z}$ given by 
$$
\delta_I(t)=\left\{\begin{array}{ll}\e(S/I)-\max\{\e(S/(I,f))\vert\,
f\in\mathcal{F}_t\}&\mbox{if }\mathcal{F}_t\neq\emptyset,\\
\e(S/I)&\mbox{if\ }\mathcal{F}_t=\emptyset.
\end{array}\right.
$$
\end{definition}

\begin{definition} 
Let $I\subseteq S$ be a graded ideal and let 
${\mathbb F}_\star$ be the minimal graded free resolution of $S/I$ as an 
$S$-module: 
\[
{\mathbb F}_\star:\ \ \ 0\rightarrow 
\bigoplus_{j}S(-j)^{\beta_{gj}}
\stackrel{}{\rightarrow} \cdots 
\rightarrow\bigoplus_{j}
S(-j)^{\beta_{1j}}\stackrel{}{\rightarrow} S
\rightarrow S/I \rightarrow 0.
\]
The {\it Castelnuovo--Mumford regularity\/}\index{Castelnuovo--Mumford}
of $S/I$, {\it regularity} of $S/I$ for short, \index{regularity} 
is defined as 
$${\rm reg}(S/I)=\max\{j-i\vert\,
\beta_{ij}\neq 0\}.$$ 
\end{definition}

The following result show the asymptotic behavior of $\delta_I$ for a particular case of graded ideals.

An ideal $I\subseteq S$ is called {\it unmixed} if all its associated primes have the same height, in other case $I$ is {\it mixed}.


\begin{theorem}[{\cite[Theorem 3.8]{MBPV}}]\label{nonincreasingDelta}
Let $I\subseteq S$ be an unmixed radical homogeneous ideal. If all the associated primes of $I$ are generated by linear forms, then there is an integer $r_0\geq 1$ such that 
$$
\delta_I(1)>\delta_I(2)>\cdots>\delta_I(r_0)=\delta_I(d)=1\
\mbox{ for }\ d\geq r_0.
$$
\end{theorem}

The integer $r_0$ where the stabilization occurs is called the {\it regularity index\/} of $\delta_I$ and is denoted by $r_I$. In Section \ref{section-mindis}, we show that one can define this index for any radical ideal.

\paragraph{\bf Local cohomology}
Let $R$ be a commutative Noetherian ring with identity and let $I$ be a homogeneous ideal generated by the forms $f_1,\ldots,f_{\ell}\in R$. Consider the \v{C}ech complex, \v{C}$^{\star}(\bar{f};R)$:
$$
 0\rightarrow R \rightarrow \bigoplus_{i}R_{f_i} \rightarrow \bigoplus_{i,j}R_{f_{i}f_{j}}
\rightarrow \cdots \rightarrow R_{f_1,\ldots,f_{\ell}}\rightarrow 0.
$$ 
where \v{C}$^{i}(\bar{f};R)=\bigoplus_{1\leq j_1\leq\ldots \leq j_i\leq \ell}R_{f_{j_1},\ldots,f_{j_i}}$ and the homomorphism in every summand is a localization map with appropriate sign. 

\begin{definition}
Let $M$ be a graded $R$-modue. The {\it $i$-th local cohomology of $M$ with support in $I$} is defined as 
\begin{center}
$H_{I}^{i}(M)= H^{i}($\v{C}$^{\star}(\bar{f};R)\otimes_{R} M)$.
\end{center}
\end{definition}

\begin{remark}\rm
Since $M$ is a graded $R$-module and $I$ is homogeneous the local cohomology module $H_{I}^{i}(M)$ is graded. 
\end{remark}

\begin{remark}\rm
If $\phi:M\rightarrow N$ is a homogeneous $R$-module homomorphism of degree $t$, then the induced $R$-module map $H_{I}^{i}(M)\rightarrow H_{I}^{i}(N)$ is homogeneous of degree $t$. 
\end{remark}

\begin{theorem}[Grothendieck's Vanishing Theorem] 
Let $M$ be an $R$-module of dimension $d$. Then, $H_{I}^{i}(M)=0$ for all $i> d$.
\end{theorem}

\begin{theorem}[Grothendieck's Non-Vanishing Theorem] 
Let $M$ be a finitely generated  $R$-module of dimension $d$. Then, $H_{I}^d(M)\neq 0$.
\end{theorem}

\begin{definition}
Let $M$ be an $R$-module with dimension $d$.  The {\it $a_i$-invariants}, $a_i(M)$, for $i=0,\ldots,d$ are defined as follows.
If $H_{\m}^{i}(M)\neq 0$, 
\begin{center}
$a_i(M)=\max\{\alpha \mid H_{\m}^{i}(M)_{\alpha}\neq 0\}$,  
\end{center} 
for $0\leq i \leq d$, where $H_{\m}^{i}(M)$ denotes the local cohomology module with support in the maximal ideal $\m$.  
If $H_{\m}^{i}(M)\neq 0$, we set $a_i(M)=-\infty.$

If $d=\dim(M)$, then, $a_d(M)$, is often just called the {\it $a$-invariant} of $M$. 
\end{definition}

The $a$-invariant, is a classical invariant  \cite{GW1}, and is closely related to the Castelnuovo-Mumford regularity.

\begin{definition}
Let $R$ be a positively graded ring and let $M$ be a finitely generated $R$-module. The {\it Castelnuovo--Mumford regularity of $M$}, $\reg(M)$, is defined as
$$
\reg(M)=\max\{a_i(M)+i \mid 1\leq i\leq d\}.
$$ 
\end{definition}  

\begin{remark}
If $M$ is a standard graded module of dimension $d$, then the $a$-invariant is related to the Castelnuovo--Mumford regularity, via the inequeality $a(M)+d\leq \reg(M)$ which is equality in the case Cohen--Macaulay.
\end{remark}

\begin{definition}
Suppose that $R$ has prime characteristic $p$. 
The Frobenius  map $F:R \rightarrow R$  is defined by  $r\mapsto r^p$.
\end{definition}

\begin{remark}
 If $R$ is reduced, $R^{1/p^{e}}$ the ring of the $p^{e}$-th roots of $R$ is well defined, and $R\subseteq R^{1/p^{e}}$.
\end{remark}

\section{Asymptotic behavior of the minimum distance function}\label{section-mindis}

In this section we prove that the minimum distance function $\delta_I$ is non-increasing. Then, the notion of regularity index of $\delta_I$ is well defined. We also find what is the stable value of the minimum distance function. We start this section establishing notation.

\begin{notation}
Given an ideal $I\subseteq S$, we set
\begin{align*}
\cA(I)&=\{\p\in\Ass_S(S/I)\; | \; \dim(S/I)=\dim(S/\p)\};\\
\cV(I)&=\{\p\in\Spec(S)\; |\; I\subseteq \p\};\\
\cD(I)&=\Spec(S)\setminus \cV(I).
\end{align*}
\end{notation}

\begin{remark}\label{RemAdditivity}
Given an  ideal $I\subseteq S$, then
$$
\e(S/I)=\sum_{\p\in\cA(I)} \lambda_{S_\p}(S_\p/IS_\p) \e(S/\p).
$$
In particular, if $I$ is radical,
$$
\e(S/I)=\sum_{\p\in\cA(I)}  \e(S/\p).
$$
\end{remark}

\begin{lemma}\label{LemmaDegAddf}
Suppose that $I$ is a radical ideal. 
Let $f\in\cF_t$ such that $\dim(S/(I,f))=\dim(S/I)$.
Then, $\cA((I,f))=\cA(I)\cap \cV(f)$.
Furthermore, 
$$\e(S/(I,f))=\sum_{\p\in \cA(I)\cap \cV(f)} \e(S/\p).$$
\end{lemma}
\begin{proof}
Let $J=(I,f)$.
Let $Q$ be an associated prime of $J$. Since $I\subseteq J$, there exists an associated prime  $\p$ of $I$
such that $\p\subseteq Q.$ If $\dim(S/Q)=\dim(S/J)=\dim(S/I),$ then $\p=Q$. Thus,
$\cA(J)\subseteq \cA(I)\cap \cV(f)$.

Let $Q\in \cA(I)\cap \cV(f)$. Then, $J\subseteq Q$ and $\dim(S/Q)=\dim(S/J).$
Then, $Q$ is a minimal prime of $S/J$. Thus, $Q\in \Ass_S(S/J)$, and so, $Q\in\cA(J).$

We note that $J$ is not necessarily radical. However, $JS_\p=IS_\p$ for every $\p\in \cA(I)\cap \cV(f).$
Thus, $\lambda_{S_\p}(S_\p/IS_\p)=1 $ for every $\p\in \cA(I)\cap \cV(f).$
Then, 
$$\e(S/(I,f))=\sum_{\p\in \cA(I)\cap \cV(f)} \e(S/\p)$$
by the additivity formula.
\end{proof}

We now show that the minimum distance function is non-increasing.

\begin{theorem}\label{ThmNonDec}
Suppose that $I$ is a radical ideal.
Then, $\delta_{I}(d)$ is a non-increasing function.
\end{theorem}
\begin{proof}

If $\cF_t= \emptyset$ for every $t\geq 1$, then $\delta_I(t)=\e(S/I)$, which is the maximum value. We note that this case is equivalent to $I$ being a prime ideal. 

We now assume that $\cF_t\neq \emptyset$ for some $t\in\NN$.  We note that in this case $\dim(S/I)\neq 0$, 
otherwise, $I=\m$ and so $\cF_t= \emptyset$.
Let $f\in\cF_t$ such that $\delta_I(t)=\e(S/I)-\e(S/(I,f))$.
Then,
\begin{align*}
\delta_I(t)&=\e(S/I)-\e(S/(I,f))\\
&=\sum_{\p\in \cA(I)} \e(S/\p) - \sum_{\p\in \cA((I,f))} \e(S/\p)\\
&=\sum_{\p\in \cA(I)} \e(S/\p) - \sum_{\p\in \cA(I)\cap \cV(f)} \e(S/\p)\\
&=\sum_{\p\in \cA(I)\cap \cD(f)} \e(S/\p).
\end{align*}
Since $I$ is radical and $\dim(S/I)>0$, we have that  $\m$ is not an associated prime.
Then, $\m f\not\subseteq I$ because $f\not\in I$. We conclude that
there exists $i=1,\ldots, n$ such that $x_if\not\in I$. In particular,
$x_i f\in \cF_{t+1}$ and 
 $\cF_{t+1}\neq \emptyset.$
Then, 
\begin{align*}
\delta_I(t+1)&\leq \sum_{\p\in \cA(I)\cap \cD(x_i f)} \e(S/\p)\\
& = \sum_{\p\in \cA(I)\cap \cD(x_i)\cap \cD(f)} \e(S/\p)\\
&\leq \sum_{\p\in \cA(I)\cap \cD(f)} \e(S/\p)=\delta_I(t).
\end{align*}
\end{proof}

Thanks to the previous theorem we have that the minimum distance function eventually stabilizes, and it has a regularity index.

\begin{definition}\label{definition-r_I}
Suppose that $I$ is a radical ideal. The \emph{regularity index of $I$}, denoted by $r_I$, is defined by
$$
r_I=\min\{s\in\NN\; |\;\ \delta_I(s)=\lim\limits_{t\to\infty} \delta_I(t)\}.
$$
\end{definition}

\begin{proposition}\label{PropStableValue}
Suppose that $I$ is a radical ideal.
Then,
$$
 \delta_{I}(t)=\min\{\;\e(S/\p)\; |\; \p\in\Ass_S(S/ I)\;\}
$$
for $t\gg 0$ if $I$ is unmixed and
$ \delta_{I}(t)=0$ for $t\gg 0$ otherwise.
\end{proposition}
\begin{proof}
We first assume that $I$ is mixed. 
Let $J_1$  be the intersection of the minimal primes of $I$ of dimension $\dim(S/I)$ and let $J_2$ be the intersection of the minimal primes of $I$ of dimension smaller than $\dim(S/I)$.
Let $f\in J_1\setminus I$.  Let $\alpha=\deg(f)$. 
In particular, $f\in\cF_\alpha.$
We note that $\dim(S/I)=\dim(S/(I,f))$ and $\cA(I)=\cA(I,f)$, and so, $\e(S/(I,f))=\e(S/I)$.
We conclude that $\delta_I(t)=0.$
Since $\delta_I$ is nondecresing by Theorem \ref{ThmNonDec}, we obtain that $\delta_I(t)=0$ for $t\geq \alpha.$

We now assume that $I$ is unmixed. Then, $\cA(I)=\Ass_S(S/I).$
If $I$ is a prime ideal, then $\delta_I(t)=\e(S/I)$ for every $t\in\NN$, and our claim follows. We assume that $\dim(S/I)>0$ and that  $I$ is not a prime ideal.
For every $f\in\cF_t,$
there exists a prime ideal $\p$ such that $f\not\in I$.
Then,
$$\e(S/I)-\e(S/(I,f))\geq \e(S/\p)\geq \min\{\;\e(S/\p)\; |\; \p\in\Ass_S(S/ I)\;\}.$$
We conclude that $\delta_I(t)\geq \min\{\;\e(S/\p)\; |\; \p\in\Ass_S(S/ I)\;\}.$
Let $\p_1,\ldots, \p_\ell$ denote the associated primes of $I$ in an order such that $e(S/\p_i)\leq \e(S/\p_j)$ for $i\geq j$.
Let $f\in \p_2\cap\ldots \cap \p_\ell\setminus I.$
Let $\alpha=\deg(f)$. We have that $f\in\cF_\alpha$.
Then, 
$\delta_I(\alpha)\leq \e(S/I)-\e(S/(I,f))=\e(S/\p_1).$ 
Since $\delta$ is nondecresing by Theorem \ref{ThmNonDec}, we obtain that $\delta_I(t)\leq\e(S/\p_1)$ for $t\geq \alpha.$
We conclude that $\delta_I(t)=\e(S/\p_1)$ for $t\geq \alpha.$ 
\end{proof}

\begin{proposition}\label{PropRegNotEqui}
Let $I$ be a mixed radical ideal.
Let $J_1$  be the intersection of the minimal primes of $I$ of dimension $\dim(S/I)$ and let $J_2$ be the intersection of the minimal primes of $I$ of dimension smaller than $\dim(S/I)$.
Then,
$r_I=\min \{t \; | \; [J_1/I]_t\neq 0\}.$
\end{proposition}
\begin{proof}
As in the proof of Proposition \ref{PropStableValue}, we have that $\delta_I(t)=0$ for $t\geq \min \{t \; | \; [J_1/I]_t\neq 0\}.$ We conclude that $r_I\leq \min \{t \; | \; [J_1/I]_t\neq 0\}.$

Let $f\in \left(\bigcup_{\p\in\Ass_S(S/I)} \p\right)\setminus I$ of degree strictly less than $\min \{t \; | \; [J_1/I]_t\neq 0\}$. Then, $f\not\in J_1$, and so, there exists a prime ideal $\p$ such that $f\not\in \p$ and $\dim(S/\p)<\dim(S/I)$. 
Then, $e(S/I)-\e(S/(I,f))\geq \e(S/\p)$. We conclude that $\delta_I(t)>0.$ Then, $r_I\geq \min \{t \; | \; [J_1/I]_t\neq 0\}.$
\end{proof}

\begin{proposition}\label{PropRegEqui}
Suppose that $I$ is an unmixed radical ideal  with associated primes $\p_1,\ldots,\p_r$ and $ \q_1,\ldots, \q_s$ such that $\e(S/\p_i)=\min\{\e(S/Q)\; |\; Q\in\Ass_S(S/I) \}$. 
Let $J_i=\left( 
\bigcap_{j\neq i}\p_j
\right)
\bigcap 
\left(\bigcap^s_{j=1}\q_j \right).$
Then,
$r_I=\min \{t \; | \; \exists i \hbox{ such that }[J_i/I]_t\neq 0\}.$
\end{proposition}
\begin{proof}
We set $\e=\e(S/\p_i).$
As in the proof of Proposition \ref{PropStableValue}, we have that $\delta_I(t)=e$ for $t\geq \min \{t \; | \; \exists i \hbox{ such that }[J_i/I]_t\neq 0\}.$ We conclude that $r_I\leq \min \{t \; | \; \exists i \hbox{ such that }[J_i/I]_t\neq 0\}.$

Let $f\in \left(\bigcup_{\p\in\Ass_S(S/I)} \p\right)\setminus I$ of degree strictly less than $\min \{t \; | \; \exists i \hbox{ such that }[J_i/I]_t\neq 0\}$. Then, $f\not\in J_i$ for every $i$, and so, either $f\not\in\q_j $ or $f$ does not belong to to different primes $\p_i$ nor $\p_j.$ In both cases, $\dim(S/I)=\dim(S/(I,f)).$
In the first case,
$\e(S/I)- \e(S/(I,f))\geq \e(S/\q_j)>e$. In the second case, 
$\e(S/I)-\e(S/(I,f))\geq 2\e>\e.$
We conclude that $\delta_I(t)\geq \e.$
 Then, $r_I\geq  \min \{t \; | \; \exists i \hbox{ such that }[J_i/I]_t\neq 0\}.$
\end{proof}

\section{Stanley--Reisner ideals associated to a shellable simplicial complex}

In this section we use the shellability condition to relate the regularity index of a Stanley--Reisner ideal of a shellable simplicial complex, $I_{\Delta}$, with the Castelnuovo--Mumford regularity.

\begin{definition}\rm
 A {\it simplicial complex\/} on a vertex set $X = \{x_1, x_2,\ldots, x_n\}$ is a collection of subsets of $X$, called {\it faces\/}, satisfying that $\{x_i\}\in \Delta$ for every $i\in [n]$ and, if $\sigma\in \Delta$ and $\theta\subseteq \sigma$ then $\theta\in \Delta$. A face of $\Delta$ not properly contained in another face of $\Delta$ is called a {\it facet\/}. 

 A face $\sigma\in \Delta$ of cardinality $\mid \sigma\mid=i+1$ has {\it dimension\/} $i$ and is called an $i$-{\it face} of $\Delta$. The {\it dimension\/} of $\Delta$ is $\dim \Delta=\max \{\dim \sigma \mid \sigma\in \Delta\}$, or if $\Delta=\{ \}$ is the void complex, which has no faces. We say that $\Delta$ is {\it pure\/} if all its facets have the same dimension.
\end{definition}

Let $\Delta$ be a simplicial complex of dimension $d$ with the vertex set $[n]=\{1,2,\ldots,n\}$, and let $\KK$ be a field. The square-free monomial ideal $I_{\Delta}$ in the polynomial ring $S = \KK[x_1,\ldots,x_n]$ is generated by the monomials $x^{\sigma} = \prod\limits_{i\in \sigma}^{} x_{i}$ which $\sigma$ is a non-face in $\Delta$.

The simplicial complex $\Delta$ is said Cohen-Macaulay when the quotient ring $\KK[\Delta] =S/I_{\Delta}$, called Stanley--Reisner ring of $\Delta$, is Cohen-Macaulay.

\begin{definition}\rm
A pure simplicial complex $\Delta$ of dimension $d$ is {\it shellable} if the facets of $\Delta$ can be order $\sigma_1,\ldots , \sigma_s$ such that 
$$
\bar{\sigma_i}\bigcap (\bigcup\limits_{j=1}^{i-1}\bar{\sigma_j})
$$
is pure of dimension $d-1$ for all $i\geq 2$. Here $\bar{\sigma_i}=\{\sigma\in \Delta\mid \sigma\subseteq \sigma_i\}$. If $\Delta$ is pure shellable, $\sigma_1,\ldots,\sigma_s$ is called a shelling.  
\end{definition}

\begin{theorem}[{\cite[Theorem 6.3.23]{Vila}}]\label{shellable-CM}
Let $\Delta$ be a simplicial complex. If $\Delta$ is pure shellable, then $\Delta$ is Cohen--Macaulay over any field $\KK$.
\end{theorem}

We come to one of our main results.

\begin{theorem}\label{ThmShellable}
Let $I=I_{\Delta}$ be the Stanley--Reisner ideal of a shellable simplicial complex, with $\dim(S/I_{\Delta})=d$. Then $r_I\leq \reg(S/I)$.  
\end{theorem}
\begin{proof}
Since $\Delta$ is shellable,  $S/I_{\Delta}$ is a Cohen--Macaulay ring by Theorem \ref{shellable-CM}. Let $\p_1,\ldots,\p_\ell$ denote the associate primes of $I$. For $1\leq i \leq \ell$, we set $R_i=S/\p_1\cap\p_2\cap\cdots \cap \p_i$ and $J_i= \p_1\cap\cdots \cap \p_i$. We have that $R_i$ is Cohen-Macaulay of dimension $d$, because $J_i$ is a shelling of $I_{\Delta}$. 

We have  the following short exact sequence;
$$
0 \longrightarrow J_{i-1}/J_i \longrightarrow R_i \longrightarrow R_{i-1} \longrightarrow 0
$$ 
for $2\leq i\leq \ell$.

Note that $\dim(R_{i-1})=d$, then $H^{i}_{\m}=0$ for all $i> d$, thus the short exact sequence induces a long exact sequence as follows: 
$$
0\to  H^{0}_{\m}(J_{i-1}/J_i)\to H^{0}_{\m}(R_i) \to  H^{0}_{\m}(R_{i-1})\to \cdots
\to  H^{d}_{\m}( J_{i-1}/J_i)\to H_{\m}^{d}(R_i)\to H^{d}_{\m}(R_{i-1})\to 0.
$$
 
    Since $R_{i-1}$ and $R_i$ are Cohen--Macaulay rings, we have
  \begin{center}
    $\reg(R_{i-1})=a_d(R_{i-1})+d$ and $\reg(R_i)=a_d(R_i)+d.$
    \end{center}
    From the exact sequence we get $\reg(J_{i-1}/J_i)\leq \reg(R_i)$. By Proposition \ref{PropRegEqui}, $r_I\leq \reg(R_i)$. Then $r_I\leq \reg(R_\ell)$, and so,  $r_I\leq \reg(S/I_{\Delta})$.
\end{proof}

\section{Results related to $F$-purity}

\begin{definition}
Let $R$ be a Noetherian ring of prime characteristic $p$, and $F:R\to R$ be the Frobenius map.
We say that $R$ is $F$-pure if for every $R$-module, $M$, we have that
$$
\xymatrix{
M\otimes_R R\ar[rr]^{1_M\otimes_R F}
&&
M\otimes_R R
           }
$$
is injective.
We say that $R$ is $F$-finite if $R$ is finitely generated as $R^p$-module.
\end{definition}

\begin{definition}
Suppose that $\KK$ has prime characteristic, $\KK$ is $F$-finite,  and
that $I$ is a radical ideal. Then, we set
\begin{itemize}
\item $\m_e=\{f\in S/I \; | \; \phi(F^e_* f)\; \forall \phi: F^e_* S/I\to S/I \}$ \cite{AE}.
\item $ b_e=\max\{t\; | \; \m^t\not\subseteq \m_e\}$.
\item $\fpt(S/I)=\lim\limits_{e\to\infty} \frac{b_e}{p^e}$ \cite{TW2004}.
\end{itemize}
\end{definition}

\begin{theorem}[{\cite[Theorem B]{DSNB}}]\label{ThmDSNB}
Suppose that $\KK$ has prime characteristic.
If $S/I$ is a $F$-pure ring,
then $a_i(S/I)\leq  -\fpt(S/I)$.
Furthermore, if $S/I$ is a Gorenstein ring, then $\reg(S/I)=d-\fpt(S/I).$
\end{theorem}

\begin{remark}\label{RmKCompstibleFPT}
Suppose that $\KK$ has prime characteristic, $\KK$ is $F$-finite,  and  that  $S/I$ is a $F$-pure ring.
Let $\p_1,\ldots, \p_\ell$ be the minimal primes of $I$, and $J_i=\bigcap_{i\neq j}\p_j$.
Then, $S/J_i$ is $F$-pure \cite[Corollary 4.8]{CentersFpurity}.
Furthermore, $\fpt(S/I)\leq \fpt(S/J_i)$ \cite[Theorem 4.7]{DSNB}, because $J_i \cdot S/I$ is a compatible ideal
for $S/I$.
\end{remark}

\begin{theorem}\label{ThmFpureDim}
Suppose that $\KK$ has prime characteristic.
If $S/I$ is a $F$-pure ring,
then  $r_I\leq \dim(S/I)$.
\end{theorem}
\begin{proof}
Let $\p_1,\ldots,\p_\ell $ be the minimal primes of $I$.
For $i=1,\ldots,\ell$, we set $J_i=\bigcap_{i\neq j}\p_j$.
We have a short exact sequence
$$
0\to J_i/I\to S/J_i\to S/I\to 0.
$$
This induces a long exact sequence
$$
0\to H^0_\m (J_i/I)\to H^0_\m ( S/J_i)\to H^0_\m (S/I)\to  H^1_\m (J_i/I)\to \ldots .
$$
Since both $S/J_i$ and $S/I$ are $F$-pure,
we have that $a_j(S/J_i)\leq 0$ and $a_j(S/I)\leq 0$ for every $j$. Then,
$a_j(J_i/I)\leq 0$ for every $j$ \cite[Proposition 2.4]{HRFpurity}.
Then, 
\begin{align*}
\min\{t\; | \; [J_i/I]_t  \neq 0\}&  \leq  \max\{\ell\; | \;\beta_{0,\ell}\neq 0\} (J_i/I)\\
& \leq \reg(J_i/I) \\
&= \max\{ a_j(J_i/I)+j \}\\
& \leq \dim(S/I).
\end{align*}
\end{proof}

\begin{theorem}\label{ThmGorFpure}
Suppose that $\KK$ has prime characteristic.
If $S/I$ is a $F$-pure ring,
then  $r_I\leq \reg(S/I)$.
\end{theorem}
\begin{proof}
We first assume that $\KK$ is $F$-finite.
Let $\p_1,\ldots,\p_\ell $ be the minimal primes of $I$.
For $i=1,\ldots,\ell$, we set $J_i=\bigcap_{i\neq j}\p_j$.
We have a short exact sequence
$$
0\to J_i/I\to S/J_i\to S/I\to 0.
$$
This induces a long exact sequence
$$
0\to H^0_\m (J_i/I)\to H^0_\m ( S/J_i)\to H^0_\m (S/I)\to H^1_\m (J_i/I)\to \ldots .
$$
Since both $S/J_i$ and $S/I$ are $F$-pure,
we have that $a_j(S/J_i)\leq -\fpt(S/J_i)$ and $a_j(S/I)\leq -\fpt(S/I)$ for every $j$. Then,
$$a_j(J_i/I)\leq \max\{-\fpt(S/J_i), -\fpt(S/I)\} \leq -\fpt(S/I) $$ for every $j$ by Theorem \ref{ThmDSNB}.
Then, 
\begin{align*}
\min\{t\; | \; [J_i/I]_t  \neq 0\}&  \leq  \max\{\ell \mid \beta_{0,\ell}(J_i/I)\neq 0\} \\
& \leq \reg(J_i/I) \\
&= \max\{ a_j(J_i/I)+j \}\\
&=\reg(S/I).
\end{align*}
The result for non  $F$-finite fields follows from taking the product $\otimes_\KK \overline{\KK}$, 
because the numerical invariants do not change after field extensions.  In addition, $F$-purity is stable for field extensions.
\end{proof}

\begin{theorem}\label{ThmSRGor}
Let $\KK$ be any field and  $I$ is a square-free monomial ideal.
Then, $r_I\leq \dim(S/I)$.
If  $S/I$ is a Gorenstein  ring,
then  $r_I\leq \reg(S/I)$.
\end{theorem}
\begin{proof}
If $\KK$ has prime characteristic, the result follows from 
Theorems \ref{ThmFpureDim} and \ref{ThmGorFpure}.

We now assume that $\KK$ has characteristic zero.
Since field extensions do not affect whether a ring is Gorenstein and their dimension, without loss of generality we can assume that $\KK=\mathbb{Q}$.
	Let $A=\ZZ[x_1,\ldots,x_n]$ and $I_A$ the monomial ideal generated by the monomials in $I$.
We have that  $r_I=r_{I_A \otimes_\ZZ \FF_p}$	by Propositions \ref{PropRegNotEqui} and \ref{PropRegEqui}, since $\dim (S/I)=\dim(A\otimes_\ZZ \QQ / I_A \otimes_\ZZ \QQ )=\dim( \FF_p [x_1,\ldots,x_n]/ I_A \otimes_\ZZ \FF_p)$.	
Then, $$r_I=r_{I_A \otimes_\ZZ \FF_p}\leq     \dim( \FF_p [x_1,\ldots,x_n]/ I_A \otimes_\ZZ \FF_p)= \reg(S/I)$$
by Theorem \ref{ThmGorFpure},  because Stanley--Reisner rings in prime characteristic are $F$-pure.

We have that  
		\begin{equation*}
			\reg_S(S/I)=\reg_{A\otimes_\ZZ \QQ}(A\otimes_\ZZ \QQ/I_A \otimes_\ZZ \QQ)=\reg_{A\otimes_\ZZ \FF_p}(A/I_A \otimes_\ZZ \FF_p)
		\end{equation*}
			and $A/J\otimes_\ZZ \FF_p$ is Gorenstein for $p\gg 0$ \cite[Theorem 2.3.5]{HH}.
Then, the result follows from Theorem \ref{ThmGorFpure}, because Stanley--Reisner rings in prime characteristic are $F$-pure.
\end{proof}

\section*{Acknowledgments} 
We thank Carlos Espinosa-Vald\'ez for comments on an earlier draft.

\bibliographystyle{alpha}
\bibliography{References}

\end{document}